\documentclass[12pt]{article}

\usepackage{amsmath}
\usepackage{amssymb}
\usepackage{amsthm}
\usepackage{amsfonts}
\usepackage{latexsym}
\usepackage{cite}
\usepackage{psfrag}
\usepackage{epsfig}
\usepackage{graphicx}
\usepackage{mathrsfs}
\usepackage{url}
\usepackage{color}
\usepackage{cancel}
\usepackage{eufrak}

\usepackage[colorlinks=true]{hyperref}
\hypersetup{urlcolor=blue, citecolor=red}

\makeatletter
\@addtoreset{equation}{section}
\makeatother

\newtheorem{thm}{Theorem}[section]
\newtheorem{lem}[thm]{Lemma}
\newtheorem{corollary}[thm]{Corollary}
\newtheorem{proposition}[thm]{Proposition}
\newtheorem{conjecture}[thm]{Conjecture}
\theoremstyle{definition}

\newtheorem{notation}[thm]{Notation}

\def\O{\mathcal{O}}
\def\P{\mathcal{P}}

\def\T{\mathcal{T}}

\def\C{\mathscr{C}}
\def\Ll{\mathscr{L}}
\def\M{\mathscr{M}}
\def\N{\mathscr{N}}

\def\PG{\mathrm{PG}}
\def\RC{\mathrm{RC}}
\def\RA{\mathrm{RA}}
\def\Tr{\mathrm{T}}
\def\IC{\mathrm{IC}}
\def\IA{\mathrm{IA}}
\def\UG{\mathrm{U\Gamma}}
\def\UnG{\mathrm{Un\Gamma}}
\def\EG{\mathrm{E\Gamma}}
\def\EnG{\mathrm{En\Gamma}}
\def\Ar{\mathrm{A}}
\def\EA{\mathrm{EA}}

\def\MM{\mathbf{M}}
\def\Pf{\mathbf{P}}

\def\F{\mathbb{F}}

\def\A{\mathfrak{A}}
\def\Lk{\mathfrak{L}}
\def\Pk{\mathfrak{P}}
\def\pk{\mathfrak{p}}

\def\a{\alpha}

\def\t{\text}
\def\db{\displaybreak[3]}
\def\dbn{\displaybreak[3]\notag}
\def\nt{\notag}

\begin{document}
\title{Twisted cubic and orbits of lines\newline in $\mathrm{PG}(3,q)$,~II
\date{}
}

\maketitle
\begin{center}
{\sc Alexander A. Davydov
\footnote{A.A. Davydov ORCID \url{https://orcid.org/0000-0002-5827-4560}}
}\\
{\sc\small Institute for Information Transmission Problems (Kharkevich institute)}\\
 {\sc\small Russian Academy of Sciences}\\
 {\sc\small Moscow, 127051, Russian Federation}\\\emph {E-mail address:} adav@iitp.ru\medskip\\
 {\sc Stefano Marcugini
 \footnote{S. Marcugini ORCID \url{https://orcid.org/0000-0002-7961-0260}},
 Fernanda Pambianco
 \footnote{F. Pambianco ORCID \url{https://orcid.org/0000-0001-5476-5365}}
 }\\
 {\sc\small Department of  Mathematics  and Computer Science,  Perugia University,}\\
 {\sc\small Perugia, 06123, Italy}\\
 \emph{E-mail address:} \{stefano.marcugini, fernanda.pambianco\}@unipg.it
\end{center}

\textbf{Abstract.}
In the projective space $\mathrm{PG}(3,q)$, we consider the orbits of lines under the stabilizer group of the twisted cubic. In the literature, lines of $\mathrm{PG}(3,q)$ are partitioned into classes, each of which is a union of line orbits.  In this paper, all classes of lines consisting of a unique orbit are found. For the remaining line types, with one exception, it is proved that they consist exactly of two or three orbits; sizes and structures of these orbits are determined. Also, the subgroups of the stabilizer group of the twisted cubic fixing lines of the orbits are obtained. Problems which remain open for one type of lines are formulated and, for $5\le q\le37$ and $q=64$, a solution is provided.



\section{Introduction}
Let $\PG(N,q)$ be the $N$-dimensional projective space over the Galois field~$\F_q$ with $q$ elements. An $n$-arc in  $\PG(N,q)$, with $n\ge N + 1\ge3$, is a set of $n$ points such that no $N +1$ points belong to the same hyperplane of $\PG(N,q)$. An $n$-arc is complete if it is not contained in an $(n+1)$-arc. For an introduction to projective geometry over finite fields see \cite{Hirs_PGFF,Segre}.

In $\PG(N,q)$, $2\le N\le q-2$, a normal rational curve is any $(q+1)$-arc projectively equivalent to the arc $\{(t^N,t^{N-1},\ldots,t^2,t,1):t\in \F_q\}\cup \{(1,0,\ldots,0)\}$. In $\PG(3,q)$, the normal rational curve is called a  \emph{twisted cubic} \cite{Hirs_PG3q}. The twisted cubic has important connections with a number of other combinatorial objects. This prompted the  twisted cubic to be widely studied, see e.g.
\cite{BonPolvTwCub,BrHirsTwCub,CasseGlynn82,CasseGlynn84,CosHirsStTwCub,GiulVincTwCub,Hirs_PG3q,KorchLanzSon,ZanZuan2010} and the references therein.

In \cite{Hirs_PG3q}, the orbits of planes, lines and points in $\PG(3,q)$ under the group $G_q$ of the projectivities fixing the twisted cubic are considered.
The sets of planes and points are partitioned into separate orbits, whereas the set of lines is partitioned into unions of orbits called {\em classes}, see Section \ref{sec_prelimin}.

The beginning of our research on the twisted cubic is the content of a paper appeared in 2020 \cite{BDMP-TwCub}  where, in collaboration with D. Bartoli, for all $q\ge2$, the structure of the \emph{point-plane} incidence matrix of $\PG(3,q)$, using point and plane orbits under $G_q$ of \cite{Hirs_PG3q}, is described. Using this matrix, we obtained in \cite{BDMP-TwCub} optimal multiple covering codes.
The results of \cite{BDMP-TwCub} allowed us to classify the cosets of the $[ q+1, q-3,5]_q 3$ generalized doubly-extended Reed-Solomon code of codimension 4 by their weight distributions \cite{DMP_CosetsRScod4}.
In the continuation of our research, aimed at determining the plane-line and point-line incidence matrices, we realized that more information on orbits was needed than that available in the literature.

In this paper, we consider the orbits of lines in $\PG(3,q)$ under the stabilizer group $G_q$ of the twisted cubic. We use the aforementioned partitions of lines into unions of orbits under $G_q$ described in \cite{Hirs_PG3q}. All classes of lines consisting of a unique orbit are found. For the remaining line classes, with the only exception of the class called $\O_6$ in \cite{Hirs_PG3q}, it is proved that they are a union of exactly two or three orbits; sizes and structures of these orbits are determined. Also, subgroups of $G_q$ fixing the lines of the orbits (apart from the lines of $\O_6$) are obtained. Problems which are still open for the class $\O_6$ of lines are formulated. For  $5\le q\le37$ and $q=64$, they are solved.

The theoretic results hold for $q\ge5$. For $q = 2,3,4$ we describe the orbits by computer search.

This paper contains all the results of \cite{DMP_OrbLine}, posted in arXiv.org. In addition, here we obtain the subgroups of $G_q$ fixing the lines of the orbits not contained in $\O_6$. We remark that the numbers and sizes of the line orbits in $\PG(3,q)$ (apart from the class $\O_6$) have been independently obtained in two papers subsequently appeared in ArXiv: \cite{BlokPelSzo} (for all $q\ge23$) and  \cite{GulLav} (for odd $q\not\equiv0\pmod3$). In the noted manuscripts methods that are different from ours are used. The representation and description of the orbits in \cite{BlokPelSzo} and \cite{GulLav} also are distinct from those of this paper. We give the orbits in a form which is convenient for the investigations in \cite{DMP_IPlLineIncarX} and \cite{DMP_PointLineInc}, where plane-line and point line incidence matrices of $\PG(3,q)$  are obtained.

The paper is organized as follows. Section \ref{sec_prelimin} contains preliminaries. In Section \ref{sec_mainres}, the main results of this paper are summarized.  In Sections \ref{sec:nul_pol}--\ref{sec:orbEA}, orbits of lines in $\PG(3,q)$ under $G_q$ and the subgroups of $G_q$ fixing the lines of the orbits are considered. In Section \ref{sec:classification}, open problems are formulated and their solutions for $5\le q\le37$ and $q=64$ are considered.

\section{Preliminaries }\label{sec_prelimin}
In this section we summarize
 the results on the twisted cubic of \cite{Hirs_PG3q} useful in this paper.

Let $\Pf(x_0,x_1,x_2,x_3)$ be a point of $\PG(3,q)$ with homogeneous coordinates $x_i\in\F_{q}$.
 Let $\F_{q}^*=\F_{q}\setminus\{0\}$, $\F_q^+=\F_q\cup\{\infty\}$. For
For $t\in\F_q^+$, let  $P(t)$ be a point such that
\begin{align}\label{eq2:P(t)}
  P(t)=\Pf(t^3,t^2,t,1)\text{ if }t\in\F_q;~~P(\infty)=\Pf(1,0,0,0).
\end{align}

Let $\C\subset\PG(3,q)$ be the \emph{twisted cubic} consisting of $q+1$ points\\
 $P_1,\ldots,P_{q+1}$  no four of which are coplanar.
We consider $\C$ in the canonical form
\begin{align}\label{eq2_cubic}
&\C=\{P_1,P_2,\ldots,P_{q+1}\}=\{P(t)\,|\,t\in\F_q^+\}.
\end{align}

A \emph{chord} of $\C$ is a line through a pair of real points of $\C$ or a pair of complex conjugate points. In the last case, it is an \emph{imaginary chord}. If the real points are distinct, it is a \emph{real chord}; if they coincide with each other, it is a \emph{tangent.}

Let $\boldsymbol{\pi}(c_0,c_1,c_2,c_3)$ $\subset\PG(3,q)$, be the plane with equation
\begin{align}\label{eq2_plane}
  c_0x_0+c_1x_1+c_2x_2+c_3x_3=0,~c_i\in\F_q.
\end{align}
The \emph{osculating plane} in the  point $P(t)\in\C$ is as follows:
\begin{align}\label{eq2_osc_plane}
&\pi_\t{osc}(t)=\boldsymbol{\pi}(1,-3t,3t^2,-t^3)\t{ if }t\in\F_q; ~\pi_\t{osc}(\infty)=\boldsymbol{\pi}(0,0,0,1).
\end{align}
 The $q+1$ osculating planes form the osculating developable $\Gamma$ to $\C$, that is a \emph{pencil of planes} for $q\equiv0\pmod3$ or a \emph{cubic developable} for $q\not\equiv0\pmod3$.

 An \emph{axis} of $\Gamma$ is a line of $\PG(3,q)$ which is the intersection of a pair of real planes or complex conjugate planes of $\Gamma$. In the last  case, it is a \emph{generator} or an \emph{imaginary axis}. If the real planes are distinct it is a \emph{real axis}; if they coincide with each other it is a \emph{tangent} to $\C$.

For $q\not\equiv0\pmod3$, the null polarity $\A$ \cite[Sections 2.1.5, 5.3]{Hirs_PGFF}, \cite[Theorem~21.1.2]{Hirs_PG3q} is given by
\begin{align}\label{eq2_null_pol}
&\Pf(x_0,x_1,x_2,x_3)\A=\boldsymbol{\pi}(x_3,-3x_2,3x_1,-x_0).
\end{align}

\begin{notation}\label{notation_1}
In future, we consider $q\equiv\xi\pmod3$ with $\xi\in\{-1,0,1\}$. Many values depend of $\xi$ or have sense only for specific $\xi$.
If it is not clear by the context, we note this by remarks.
The following notation is used.
\begin{align*}
  &G_q && \t{the group of projectivities in } \PG(3,q) \t{ fixing }\C;\db  \\
&A^{tr}&&\t{the transposed matrix of }A;\db \\
&\#S&&\t{the cardinality of a set }S;\db\\
&\overline{AB}&&\t{the line through the points $A$ and }B;\db\\
&\triangleq&&\t{the sign ``equality by definition"}.\db\\
&P_t && \t{the point } P(t) \t{ of } \C \t{ with } t\in\F_q^+, \t{ cf. } \eqref{eq2:P(t)}, \eqref{eq2_cubic};\db  \\
&\T_t  && \t{the tangent line to } \C \t{ at the point } \P_t;\db  \\
&&&\t{\textbf{Types $\pi$ of planes:}}\db\\
&\Gamma\t{-plane}  &&\t{an osculating plane of }\Gamma;\db \\
&d_\C\t{-plane}&&\t{a plane containing \emph{exactly} $d$ distinct points of }\C,~d=0,2,3;\db \\
&\overline{1_\C}\t{-plane}&&\t{a plane not in $\Gamma$ containing \emph{exactly} 1 point of }\C;\db \\
&\Pk&&\t{the list of possible types $\pi$ of planes},~\Pk\triangleq\{\Gamma,2_\C,3_\C,\overline{1_\C},0_\C\};\db\\
&\pi\t{-plane}&&\t{a plane of the type }\pi\in\Pk; \db\\
&\N_\pi&&\t{the orbit of $\pi$-planes under }G_q,~\pi\in\Pk.\db\\
&&&\t{\textbf{Types $\lambda$ of lines with respect to  the twisted cubic $\C$:}}\db\\
&\RC\t{-line}&&\t{a real chord  of $\C$;}\db \\
&\RA\t{-line}&&\t{a real axis of $\Gamma$ for }\xi\ne0;\db \\
&\Tr\t{-line}&&\t{a tangent to $\C$};\db \\
&\IC\t{-line}&&\t{an imaginary chord  of $\C$;}\db \\
&\IA\t{-line}&&\t{an imaginary axis of $\Gamma$ for }\xi\ne0;\db \\
&\UG&&\t{a non-tangent unisecant in a $\Gamma$-plane;}\db \\
&\t{Un$\Gamma$-line}&&\t{a unisecant not in a $\Gamma$-plane (it is always non-tangent);}\db \\
&\t{E$\Gamma$-line}&&\t{an external line in a $\Gamma$-plane (it cannot be a chord);}\db \\
&\t{En$\Gamma$-line}&&\t{an external line, other than a chord, not in a $\Gamma$-plane;}\db \\
&\Ar\t{-line}&&\t{the axis of $\Gamma$ for }\xi=0\db\\
&&&\t{(it is the single line of intersection of all the $q+1~\Gamma$-planes)};\db \\
&\EA\t{-line}&&\t{an external line meeting the axis of $\Gamma$ for }\xi=0;\db\\
&\Lk&&\t{the list of possible types $\lambda$ of lines},\db\\
&&&\Lk\triangleq\{\RC,\RA,\Tr,\IC,\IA,\UG,\UnG,\EG,\EnG\}\t{ for }\xi\ne0,\db\\
&&&\Lk\triangleq\{\RC,\Tr,\IC,\UG,\UnG, \EnG,\Ar,\EA\}\t{ for }\xi=0;\db\\
&\lambda\t{-line}&&\t{a line of the type }\lambda\in\Lk;\db\\
&\O_\lambda&&\t{the union (class) of all orbits of $\lambda$-lines under }G_q,~\lambda\in\Lk.\db\\
&&&\t{\textbf{Types of points with respect to the twisted cubic $\C$:}}\db\\
&\C\t{-point}&&\t{a point  of }\C;\db\\
&\mu_\Gamma\t{-point}&&\t{a point  off $\C$ lying on \emph{exactly} $\mu$ distinct osculating planes,}\db\\
&&&\mu_\Gamma\in\{0_\Gamma,1_\Gamma,3_\Gamma\}\t{ for }\xi\ne0,~\mu_\Gamma\in\{(q+1)_\Gamma\}\t{ for }\xi=0;\db\\
&\Tr\t{-point}&&\t{a point  off $\C$  on a tangent to $\C$ for }\xi\ne0;\db\\
&\t{TO-point}&&\t{a point  off $\C$ on a tangent and one osculating plane for }\xi=0;\db\\
&\RC\t{-point}&&\t{a point  off $\C$  on a real chord;}\db\\
&\IC\t{-point}&&\t{a point  on an imaginary chord (it is always off $\C$).}
\end{align*}
\end{notation}

The following theorem summarizes results from \cite{Hirs_PG3q} useful in this paper.
\begin{thm}\label{th2_Hirs}
\emph{\cite[Chapter 21]{Hirs_PG3q}} The following properties of the twisted cubic $\C$ of \eqref{eq2_cubic} hold:
  \begin{align}
  &\textbf{\emph{(i)}} \t{ The group $G_q$ acts triply transitively on }\C.
  ~ G_q\cong PGL(2,q)~\t{ for }q\ge5;\dbn \\
&\t{ The matrix $\MM$ corresponding to a projectivity of $G_q$ has the general form}\dbn\\
& \label{eq2_M} \mathbf{M}=\left[
 \begin{array}{cccc}
 a^3&a^2c&ac^2&c^3\\
 3a^2b&a^2d+2abc&bc^2+2acd&3c^2d\\
 3ab^2&b^2c+2abd&ad^2+2bcd&3cd^2\\
 b^3&b^2d&bd^2&d^3
 \end{array}
  \right],a,b,c,d\in\F_q,~ ad-bc\ne0.
\end{align}

\textbf{\emph{(ii)}} Under $G_q$, $q\ge5$, there are five orbits of planes and five orbits of points.

\textbf{\emph{(a)}} For all $q$, the orbits $\N_i$ of planes are as follows:
\begin{align}\label{eq2_plane orbit_gen}
   &\N_1=\N_\Gamma=\{\Gamma\t{-planes}\},~~~~\#\N_\Gamma=q+1;\db\\
   &\N_{2}=\N_{2_\C}=\{2_\C\t{-planes}\}, ~\#\N_{2_\C}=q^2+q;\dbn \\
 &\N_{3}=\N_{3_\C}=\{3_\C\t{-planes}\},~  \#\N_{3_\C}=(q^3-q)/6;\dbn\\
 &\N_{4}=\N_{\overline{1_\C}}=\{\overline{1_\C}\t{-planes}\},~\#\N_{\overline{1_\C}}=(q^3-q)/2;\dbn\\
 &\N_{5}=\N_{0_\C}=\{0_\C\t{-planes}\},~\#\N_{0_\C}=(q^3-q)/3.\nt
 \end{align}

\textbf{\emph{(b)}} For $q\not\equiv0\pmod 3$, the orbits $\M_j$ of points are as follows:
\begin{align*}
&\M_1=\{\C\t{-points}\},~ \M_2=\{\Tr\t{-points}\},~\M_3=\{3_\Gamma\t{-points}\},,\db\\
&\M_4=\{1_\Gamma\t{-points}\},~\M_5=\{0_\Gamma\t{-points}\}.\db\\
&\t{Also, if } q\equiv1\pmod 3 \t{ then } \M_{3}\cup\M_{5}=\{\RC\t{-points}\}, ~\M_{4}=\{\IC\t{-points}\};\db\\
 &\t{if } q\equiv-1\pmod 3\t{ then }\M_{3}\cup\M_{5}=\{\IC\t{-points}\},~ \M_{4}=\{\RC\t{-points}\}.
\end{align*}

\textbf{\emph{(c)}} For $q\equiv0\pmod 3$, the orbits $\M_j$ of points are as follows:
\begin{align*}
&\M_1=\{\C\t{-points}\},~\M_2=\{(q+1)_\Gamma\t{-points}\},~\M_3=\{\t{\emph{TO}-points}\},\db\\
&\M_4=\{\RC\t{-points}\},~ \M_5=\{\IC\t{-points}\}.
\end{align*}

 \textbf{\emph{(iii)}} For $q\not\equiv0\pmod3$, the null polarity $\A$ \eqref{eq2_null_pol} interchanges $\C$ and $\Gamma$ and their corresponding chords and axes.

 \textbf{\emph{(iv)}} The lines of $\PG(3,q)$ can be partitioned into classes called $\O_i$ and $\O'_i$, each of which is a union of orbits under $G_q$.
  \begin{align}
  &\hspace{0.6cm}\textbf{\emph{(a)}}~ q\not\equiv0\pmod3,~ q\ge5, ~\O'_i=\O_i\A,~ \#\O'_i=\#\O_i,~i=1,\ldots,6.\dbn\\
  &\O_1=\O_\RC=\{\RC\t{-lines}\},~\O'_1=\O_\RA=\{\RA\t{-lines}\},\db\label{eq2_classes line q!=0mod3}\\
  &\#\O_\RC=\#\O_\RA=(q^2+q)/2;\dbn\\
  &\O_2=\O'_2=\O_\Tr=\{\Tr\t{-lines}\},~\#\O_\Tr=q+1;\dbn \\
  &\O_3=\O_\IC=\{\IC\t{-lines}\},~\O'_3=\O_\IA=\{\IA\t{-lines}\},\dbn\\
  &\#\O_\IC=\#\O_\IA=(q^2-q)/2;\dbn\\
  &\O_4=\O'_4=\O_\UG=\{\UG\t{-lines}\},~\#\O_\UG=q^2+q;\dbn\\
  &\O_5=\O_\UnG=\{\UnG\t{-lines}\},\O'_5=\O_\EG=\{\EG\t{-lines}\},\dbn\\
  &\#\O_\UnG=\#\O_\EG=q^3-q;\dbn\\
  &\O_6=\O'_6=\O_\EnG=\{\EnG\t{-lines}\},~\#\O_\EnG=(q^2-q)(q^2-1).\nt
     \end{align}
  For $q>4$ even, the lines in the regulus complementary to that of the tangents form an orbit of size $q+1$ contained in $\O_4=\O_\UG$.
  \begin{align}
  &\textbf{\emph{(b)}}~q\equiv0\pmod3,~q>3.\dbn\\
  &\t{Classes }\O_1,\ldots,\O_6\t{ are as in \eqref{eq2_classes line q!=0mod3}};~\O_7=\O_\Ar=\{\Ar\t{-line}\},~\#\O_\Ar=1;\label{eq2_classes line q=0mod3}\\
  &\O_8=\O_\EA=\{\EA\t{-lines}\},~\#\O_\EA=(q+1)(q^2-1). \nt
     \end{align}

 \textbf{\emph{(v)}} The following properties of chords and axes hold.

 \textbf{\emph{(a)}}  For all $q$, no two chords of $\C$ meet off $\C$.

 \phantom{\textbf{\emph{(a)}}} Every point off $\C$ lies on exactly one chord of $\C$.

 \textbf{\emph{(b)}}       Let $q\not\equiv0\pmod3$.

 \phantom{\textbf{\emph{(b)}}}  No two axes of $\Gamma$ meet unless they lie in the same plane of $\Gamma$.

 \phantom{\textbf{\emph{(b)}}}  Every plane not in $\Gamma$ contains exactly one axis of $\Gamma$.

  \textbf{\emph{(vi)}} For $q>2$, the unisecants of $\C$ such that every plane through such a unisecant meets $\C$ in at most one point other than the point of contact are, for $q$ odd, the tangents, while for $q$ even, the tangents and the unisecants in the complementary regulus.
\end{thm}

\section{The main results}\label{sec_mainres}
Throughout the paper, we consider orbits of lines and planes under $G_q$.

From now on, we consider $q\ge5$ apart from Theorem \ref{th3:q=2 3 4}.

Theorems \ref{th3_main_res}, \ref{th3_main_res_stab} summarize the results of  Sections \ref{sec:nul_pol}--\ref{sec:orbEA}.

\begin{thm}\label{th3_main_res}
Let $q\ge5$, $q\equiv\xi\pmod3$. Let notations be as in Section \emph{\ref{sec_prelimin}} including  Notation~\emph{\ref{notation_1}}. For line orbits under $G_q$ the following holds.
\begin{description}
  \item[(i)] The following classes of lines consist of a single orbit:\\
   $\O_1=\O_\RC=\{\RC\t{-lines}\}$,
  $\O_2=\O_\Tr=\{\Tr\t{-lines}\}$, and\\
   $\O_3=\O_\IC=\{\IC\t{-lines}\}$,  for all~$q$;\\
   $\O_4=\O_\UG=\{\UG\t{-lines}\}$, for odd $q$;\\
    $\O_5=\O_\UnG=\{\UnG\t{-lines}\}$ and $\O'_5=\O_\EG=\{\EG\t{-lines}\}$, for even $q$;\\
     $\O_1'=\O_\RA=\{\RA\t{-lines}\}$ and $\O_3'=\O_\IA=\{\IA\t{-lines}\}$,
  for $\xi\ne0$;\\
   $\O_7=\O_\Ar=\{\Ar\t{-lines}\}$, for $\xi=0$.

  \item[(ii)]
Let $q\ge8$ be even. The non-tangent unisecants in a $\Gamma$-plane \emph{(}i.e. $\UG$-lines, class $\O_4=\O_\UG$\emph{)} form two orbits of size $q+1$ and $q^2-1$. The orbit of size $q+1$  consists of the lines in the regulus complementary to that of the tangents. Also, the $(q+1)$-orbit and $(q^2-1)$-orbit can be represented in the form $\{\ell_1\varphi|\varphi\in G_q\}$ and $\{\ell_2\varphi|\varphi\in G_q\}$, respectively, where $\ell_j$ is a line such that  $\ell_1=\overline{P_0\Pf(0,1,0,0)}$, $\ell_2=\overline{P_0\Pf(0,1,1,0)}$, $P_0=\Pf(0,0,0,1)\in\C$.

  \item[(iii)] Let $q\ge5$ be odd.
 The non-tangent unisecants not in a $\Gamma$-plane \emph{(}i.e. $\UnG$-lines, class $\O_5=\O_\UnG$\emph{)} form  two orbits of size $(q^3-q)/2$. These orbits can be represented in the form $\{\ell_j\varphi|\varphi\in G_q\}$, $j=1,2$, where $\ell_j$ is a line such that $\ell_1=\overline{P_0\Pf(1,0,1,0)}$,  $\ell_2=\overline{P_0\Pf(1,0,\rho,0)}$, $P_0=\Pf(0,0,0,1)\in\C$, $\rho$ is not a square.

   \item[(iv)] \looseness -1
Let $q\ge5$ be odd. Let $q\not\equiv0\pmod 3$. The external lines in a $\Gamma$-plane   \emph{(}class $\O_5'=\O_\EG$\emph{)} form two orbits of size $(q^3-q)/2$. These orbits can be represented in the form $\{\ell_j\varphi|\varphi\in G_q\}$, $j=1,2$, where $\ell_j=\pk_0\cap\pk_j$ is the intersection line of planes $\pk_0$ and $\pk_j$ such that
           $\pk_0=\boldsymbol{\pi}(1,0,0,0)=\pi_\t{\emph{osc}}(0)$, $\pk_1=\boldsymbol{\pi}(0,-3,0,-1)$,  $\pk_2=\boldsymbol{\pi}(0,-3\rho,0,-1)$, $\rho$ is not a square, cf. \eqref{eq2_plane}, \eqref{eq2_osc_plane}.

   \item[(v)]
Let $q\equiv0\pmod 3,\; q\ge9$. The external lines meeting the axis of $\Gamma$ \emph{(}i.e. $\EA$-lines, class $\O_8=\O_\EA$\emph{)} form three orbits of size $q^3-q$, $(q^2-1)/2$, $(q^2-1)/2$. The $(q^3-q)$-orbit and the two $(q^2-1)/2$-orbits can be represented in the form $\{\ell_1\varphi|\varphi\in G_q\}$ and $\{\ell_j\varphi|\varphi\in G_q\}$, $j=2,3$, respectively, where $\ell_j$ are lines such that $\ell_1=\overline{P_0^\Ar\Pf(0,0,1,1)}$,  $\ell_2=\overline{P_0^\Ar\Pf(1,0,1,0)}$, $\ell_3=\overline{P_0^\Ar\Pf(1,0,\rho,0)}$, $P_0^\Ar=\Pf(0,1,0,0)$, $\rho$ is not a square.
\end{description}
\end{thm}

\begin{thm}\label{th3_main_res_stab}
Let $q\ge5$. Let notations be as in Section \emph{\ref{sec_prelimin}} including  Notation~\emph{\ref{notation_1}}. For line orbits under $G_q$, apart from the orbits contained in the class $\O_6$, we give at least an example of line belonging to each orbit and the description of the subgroup of $G_q$ fixing the line.
These results are obtained in the following theorems and corollaries:\\
for $\T\t{-lines}$ in Theorem \ref{th5:T-lines};
for $\RC\t{-lines}$  in Theorem \ref{th5:RC-lines};\\
for $\RA\t{-lines}$, $q\not\equiv0\pmod 3$,  in Corollary \ref{cor5:O1'};
for $\IC\t{-lines}$  in Theorem \ref{th5:O3orbit};\\
for $\IA\t{-lines}$, $q\not\equiv0\pmod 3$,  in Corollary \ref{cor5_O3'};\\
for $\UG\t{-lines}$, $q$ odd, in Theorem \ref{th5:O4UGorbitOdd};
for $\UG\t{-lines}$, $q$ even, in Theorem \ref{th5:O4UGorbitEven};\\
for $\UnG\t{-lines}$, $q$ even,  in Theorem \ref{th5:O5UnGorbitEven};
for $\UnG\t{-lines}$, $q$ odd,  in Theorem~\ref{th5:O5UnGorbitOdd};\\
for $\EG\t{-lines}$, $q\not\equiv0\pmod 3$,  $q$ even, in Corollary \ref{cor5:O5'EGeven};\\
for $\EG\t{-lines}$, $q\not\equiv0\pmod 3$,  $q$ odd, in Corollary \ref{cor5:O5'EGodd};\\
for $\EA\t{-lines}$,  $q\equiv0\pmod 3$, in Theorem~\ref{th5:EAorbit}.
\end{thm}

In the previous theorem, we give the subgroup of $G_q$ fixing a line, for all line orbits of $G_q$, apart from the orbits contained in the class $\O_6$. We recall that the subgroups of $G_q$ fixing lines belonging to the same orbit of $G_q$ are conjugate in $G_q$ and therefore isomorphic.

Theorem~\ref{th3:q=2 3 4} is obtained by an exhaustive computer search using the symbol calculation system Magma~\cite{Magma}.

\begin{thm}\label{th3:q=2 3 4}
Let notations be as in Section \emph{\ref{sec_prelimin}} including  Notation~\emph{\ref{notation_1}}.
 Let $\mathbf{Z}_n$ be a cyclic group of order $n$. Let $\mathbf{S}_n$ by the symmetric group of degree $n$.
\begin{description}
    \item[(i)] Let $q=2$. The group $G_2\cong\mathbf{S}_3\mathbf{Z}_2^3$ contains $8$ subgroups isomorphic to $PGL(2,2)$ divided into two conjugacy classes. For one of these subgroups, the matrices corresponding to the projectivities of the subgroup assume the form described by \eqref{eq2_M}. For this subgroup (and only for it) the line orbits under it are the same as in Theorem \emph{\ref{th3_main_res}} for $q\equiv-1\pmod3$.

    \item[(ii)] Let $q=3$. The group $G_3\cong\mathbf{S}_4\mathbf{Z}_2^3$ contains $24$ subgroups isomorphic to $PGL(2,3)$ divided into four conjugacy classes. For one of these subgroups, the matrices corresponding to the projectivities of the subgroup assume the form described by \eqref{eq2_M}. For this subgroup (and only for it) the line orbits under it are the same as in Theorem \emph{\ref{th3_main_res}} for $q\equiv0\pmod3$.

    \item[(iii)] Let $q=4$. The group $G_4\cong\mathbf{S}_5\cong P\Gamma L(2,4)$ contains one subgroup isomorphic to $PGL(2,4)$. The matrices corresponding to the projectivities of this subgroup assume the form described by \eqref{eq2_M} and for this subgroup, the line orbits under it are the same as in Theorem \emph{\ref{th3_main_res}} for $q\equiv1\pmod3$.
\end{description}
\end{thm}


\section{The null polarity $\A$ and orbits under $G_q$ of lines in $\PG(3,q)$}\label{sec:nul_pol}

  Let   $\MM$ be the general form of a matrix corresponding to a projectivity of $G_q$ given by \eqref{eq2_M}. Using the system of symbolic computation  Maple \cite{Maple}, we obtain its inverse matrix $\MM^{-1}$:
  \begin{align}\label{eq5_inversM}
&    \MM^{-1}=\left[
 \begin{array}{cccc}
 d^3A^{-1}&cd^2A^{-1}&c^2 dA^{-1}&c^3A^{-1}\\
 3bd^2A^{-1}&d(ad+2bc)A^{-1}&c(2ad+bc)A^{-1}&3ac^2B^{-1}\\
 3b^2dB^{-1}&b(2ad+bc)B^{-1}&a(ad+2bc)B^{-1}&3a^2cA^{-1}\\
 b^3A^{-1}&ab^2A^{-1}&a^2bA^{-1}&a^3A^{-1}
 \end{array}
 \right],\db\\
 &A=a^3d^3-b^3 c^3+3ab^2c^2d-3a^2bcd^2,~B=(a^2d^2-2abcd+b^2c^2)(ad-bc).\nt
  \end{align}

\begin{lem}\label{pol_comm}
 Let $q\not\equiv0\pmod 3$. Let $\A$ be the null polarity \emph{\cite[Theorem 21.1.2]{Hirs_PG3q}} given by \eqref{eq2_null_pol}. Let $P=\Pf(x_0,x_1,x_2,x_3)$ be a point of $\PG(3,q)$, $P\A$ be its polar plane, and $\Psi$ be a projectivity belonging to $G_q$.
    Then
  \begin{align}\label{eq5_FU=UF}
   (P\A)\Psi=(P\Psi)\A.
  \end{align}
\end{lem}
\begin{proof}Using the matrices $\MM$ and $\MM^{-1}$ of \eqref{eq2_M} and \eqref{eq5_inversM}, respectively, we define $x_i'$ and $\overline{c_i}$ as follows:\\
$
 [x_0',x_1',x_2',x_3']=[x_0,x_1,x_2,x_3]\mathbf{M},\,[\overline{c_0},\overline{c_1},\overline{c_2},\overline{c_3}]^{tr}=\mathbf{M}^{-1}
[c_0,c_1,c_2,c_3]^{tr}.
$
Then it is well known (see e.g. \cite[Chapter 4, Note 23]{Cassebook}) that:
\begin{align*}
\boldsymbol{\pi}(c_0,c_1,c_2,c_3)\Psi=\boldsymbol{\pi}(\overline{c_0},\overline{c_1},\overline{c_2},\overline{c_3}).
\end{align*}

By above and by \eqref{eq2_null_pol}, \eqref{eq2_M}, \eqref{eq5_inversM}, we have $P\Psi=\Pf(x_0',x_1',x_2',x_3');$
\begin{align*}
&(P\Psi)\A=\boldsymbol{\pi}(x_3',-3x_2',3x_1',-x_0');~~P\A=\boldsymbol{\pi}(x_3,-3x_2,3x_1,-x_0);\db \\
&(P\A)\Psi=\boldsymbol{\pi}(v_0,v_1,v_2,v_3),~[v_0,v_1,v_2,v_3]^{tr}=\mathbf{M}^{-1}[x_3,-3x_2,3x_1,-x_0]^{tr}.
\end{align*}
By direct symbolic computation using the system Maple, we verified that
\begin{align*}
   \mathbf{M}^{-1}[x_3,-3x_2,3x_1,-x_0]^{tr}=[x_3',-3x_2',3x_1',-x_0']^{tr}. \hspace{3cm}
\end{align*}
\end{proof}

\begin{thm}\label{th5_null_pol}
 Let $q\not\equiv0\pmod 3$. Let $\Ll$ be an orbit of lines under~$G_q$. Then $\Ll\mathfrak{A}$ also is an orbit of lines under~$G_q$.
\end{thm}
\begin{proof}
 We take the line $\ell_1$ through the points $P_1$ and $P_2$ of $\PG(3,q)$ and a projectivity $\Psi\in G_q$. Let $\ell_2$ be the line through
$Q_1 = P_1\Psi$ and $Q_2 = P_2\Psi$. Then $\ell_1$ and $\ell_2$  belong to the same orbit
and $\ell_2=\ell_1\Psi$.

 We show that $\ell_2\A=(\ell_1\A)\Psi$.
Let $\pk_i=P_i\A,\,\pk_i'=Q_i\A,\, i=1,2$. By~\eqref{eq5_FU=UF},
\begin{align*}
 &\pk_1'=Q_1\A=(P_1\Psi)\A=(P_1\A)\Psi=\pk_1\Psi,\db\\
 &\pk_2'=Q_2\A=(P_2\Psi)\A=(P_2\A)\Psi=\pk_2\Psi.
\end{align*}
So,  we have  $\ell_2\A=\pk_1' \cap \pk_2' = \pk_1 \Psi \cap \pk_2 \Psi =  (\pk_1 \cap \pk_2)\Psi = (\ell_1\A)\Psi$.
\end{proof}

\begin{lem}\label{lemma_same_stab}
 Let $q\not\equiv0\pmod 3$. Let $\ell$ be a line and let $\ell'= \ell \mathfrak{A}$. Then the subgroup of $G_q$ fixing $\ell$ is also the subgroup of $G_q$ fixing $\ell'$.
\end{lem}
\begin{proof}
 Let $\ell=\overline{Q_1,Q_2}$, $Q_1, Q_2 \in \PG(3,q) $ and let $\Psi\in G_q$  such that $\ell \Psi = \ell$. Then   $\ell=\overline{Q_1\Psi,Q_2\Psi}$ also. Let $\pk_i=Q_i\A$,   $\pk_i'=Q_i\Psi\A, i = 1,2$.  By definition, $\ell' = \pk_1 \cap  \pk_2 = \pk'_1 \cap  \pk'_2$.

 By Lemma \ref{pol_comm},
$\ell'\Psi = (\pk_1 \cap  \pk_2)\Psi = \pk_1\Psi \cap  \pk_2\Psi = Q_1\A\Psi \cap  Q_2\A\Psi =  Q_1\Psi \A \cap Q_2\Psi \A= \pk'_1 \cap  \pk'_2= \ell'.$
\end{proof}

\section{Orbits under $G_q$ of chords and tangents of the cubic $\C$ and axes of the osculating developable $\Gamma$
 }\label{sec:orbChAx}

\begin{notation}\label{notation_3}
In addition to Notation \ref{notation_1}, the following notation is used.
\begin{align*}
&G_q^{\ell}&&  \t{the subgroup of } G_q \t{ fixing the line }\ell;\db  \\
&a,b,c,d && \t{the elements of the matrix of the form } \eqref{eq2_M}\db  \\ &&&\t{corresponding to a projectivity of } G_q.
\end{align*}
\end{notation}

\begin{lem}\label{eqTangents}
The tangent $\T_t$ to $\C$ at the point $P_t$ has the following equations:
\begin{align*}
&  \T_{\infty} : \begin{cases}
												x_2=0 \\
												x_3=0
            								\end{cases};~~~~  \T_1 :\begin{cases}
												 x_0 = x_1 + x_2- x_3 \\
												x_0 = 3x_2 -2x_3
            								\end{cases};\\
& \T_t,~ t\in\F_q\setminus\{1\}: \begin{cases}
												 x_0 = tx_1 +t^2 x_2-t^3 x_3\\
												x_1 = tx_0 + (2t-3t^3) x_2 + (2t^4-t^2) x_3
            								\end{cases}.
\end{align*}
\end{lem}
\begin{proof}
The point $P_t= \Pf (t^3,t^2,t,1)$, $ t\in\F_q$,  can be considered as an affine point with respect to the infinite plane $x_3=0$. Then the slope of the tangent line to $\C$ at $P_t$ is obtained by deriving the parametric equation of $\C$ and is $(3t^2,2t,1)$. It means that $\T_t$ contains the infinite point $Q_t = \Pf (3t^2,2t,1,0)$.
The planes $\pk_1$ of equation $ x_0 = tx_1 +t^2 x_2-t^3 x_3$ and
												$\pk_2$ of equation $ x_1 = tx_0 + (2t-3t^3) x_2 + (2t^4-t^2) x_3$ contain both the points $P_t$ and $Q_t$.

However, if $t=1$,  $\pk_1 = \pk_2$, so we consider $\pk_3$ of equation $ x_0 = 3x_2 -2x_3$ as second plane containing both $P_t$ and $Q_t$.
In particular $\T_0$ has equations  $ x_0 = 0, x_1 =0 $.

Now consider the projectivity $\Psi$  of equations  $ x_0' = x_3,  x_1' = x_2,  x_2' = x_1,  x_3' = x_0$. Then
$P_0 \Psi = P_{\infty}$, $P_{\infty} \Psi = P_0$, and  $P_{t} \Psi = P_{1/t}$ if $t \neq 0$. It means that $\Psi \in G_q$ and $\T_{\infty} = \T_0 \Psi$ has equations  $ x_2 = 0, x_3 =0 $.
\end{proof}

\begin{thm}\label{th5:T-lines}
For any $q\ge5$, the tangents \emph{(}i.e.\ $\Tr$-lines, class $\O_2=\O_\Tr$\emph{)} to the twisted cubic $\C$ \eqref{eq2_cubic} form an orbit under $G_q$ and the group $G_q$ acts triply transitively on this orbit.
The subgroup of $G_q$ fixing a tangent has size $q(q-1)$.
Let $\T_\infty$ be the tangent to $\C$ at $P_\infty$. An element $\Psi$ of $G_q^{\T_\infty}$ has a matrix of the form:
 \begin{equation*}\label{S5:stab_mat_tang}
   \MM^{\T_\infty}=\left[
 \begin{array}{cccc}
 1&0&0&0\\
 3b&d&0&0\\
 3b^2&2bd&d^2&0\\
 b^3&b^2d&bd^2&d^3
 \end{array}
  \right], \;\; b \in \F_q, d \in \F_q^*.
\end{equation*}
\end{thm}

\begin{proof}
The triple transitive action  of  $G_q$ on $\C$ implies that the tangents belong to the same orbit.
Then by Theorem \ref{th2_Hirs}(iv)(a) and by \cite[Lemma 2.44(ii)]{Hirs_PGFF},
the subgroup of $G_q$ fixing a tangent has size $q(q-1)$.
Consider the tangent $\T_\infty$ to $\C$ at $P_\infty$. By Lemma \ref{eqTangents},  $\T_\infty$ has equations $x_2 = 0, x_3 = 0$ and therefore $\T_\infty=\overline{ P_\infty, \Pf(0,1,0,0)}$.
Let  $\Psi\in G_q$ and let $\MM$ be a matrix corresponding to $\Psi$. By \eqref{eq2_M},
\begin{equation*}
[1,0,0,0]\MM=[a^3,a^2c,ac^2,c^3].
\end{equation*}
Then $P_\infty \Psi\in \T_\infty$ implies $c = 0$ and $a \neq 0$. As, by \eqref{eq2_M}, $ad-bc \neq 0$, also $d \neq 0$.
On the other hand, if $c=0$,
\begin{equation*}
[0,1,0,0]\MM=[3a^2b,a^2d,0,0],
\end{equation*}
so $\Pf(0,1,0,0) \Psi\in \T_\infty$.
As $\MM$ is defined up to a factor of proportionality, we can put $a=1$.
\end{proof}

A chord of $\C$ through the points $P(t_1)$ and $P(t_2)$ is a line joining either a pair of real points of
$\C$ or a pair of complex conjugate points. According to \cite[Section 21.5]{Hirs_PG3q}, the coordinate vector (see \cite[Section 15.2]{Hirs_PG3q}) of a chord is \cite[equation (21.5)]{Hirs_PG3q}):
\begin{equation}\label{eq1:chord}
 L_{\t{ch}}=(a_2^2, a_l a_2, a_1^2-a_2, a_2, -a_1, 1)
\end{equation}
where $a_1 =t_ 1 + t_ 2$, $a_2 =t_1 t_ 2$. Moreover, if $x^2- a_1x + a_2$ has 2, 1, or 0  roots in $\F_q$ then \eqref{eq1:chord} gives, respectively, a real chord, a tangent, or an imaginary chord.

\begin{thm}\label{th5:RC-lines}
For any $q\ge5$, the real chords \emph{(}i.e.\ $\RC$-lines, class $\O_1=\O_\RC$\emph{)} of the twisted cubic $\C$ \eqref{eq2_cubic} form an orbit under $G_q$. The subgroup of $G_q$ fixing a real chord has size $2(q-1)$. Let $\ell$ be the real chord $\overline{P_0 P_\infty}$. An element $\Psi$ of $G_q^{\ell}$ has a matrix of the form: \begin{equation}\label{S5:stab_mat_cord}
   \MM^{\ell}=\left[
 \begin{array}{cccc}
 1&0&0&0\\
 0&d&0&0\\
 0&0&d^2&0\\
 0&0&0&d^3
 \end{array}
  \right], d \in \F_q^*,  \text{ or }  \;  \MM^{\ell}=\left[
 \begin{array}{cccc}
 0&0&0&1\\
 0&0&b&0\\
 0&b^2&0&0\\
 b^3&0&0&0
 \end{array}
  \right], b \in \F_q^*.
\end{equation}
\end{thm}

\begin{proof}
The triple transitive action  of  $G_q$ on C implies that the real chords belong to the same orbit. Then by Theorem \ref{th2_Hirs}(iv)(a) and by \cite[Lemma 2.44(ii)]{Hirs_PGFF}, the subgroup of $G_q$ fixing a real chord has size $2(q-1)$.
Consider the real chord  $\ell=\overline{P_0 P_\infty}$. The line $\ell$ has equations $x_1 = 0, x_2 = 0$.
Let  $\Psi\in G_q$ and let $\MM$ be a matrix corresponding to $\Psi$. By \eqref{eq2_M},
\begin{equation*}\label{ThS5:stab_eq1}
[0,0,0,1]\MM=[b^3,b^2d,bd^2,d^3].
\end{equation*}
Then $P_0 \Psi\in \ell$ implies $b = 0$ or $d = 0$.
On the other hand,
\begin{equation*}
[1,0,0,0]\MM=[a^3,a^2c,ac^2,3c^3].
\end{equation*}
Then $P_\infty \Psi\in \ell$ implies $a = 0$ or $c = 0$.
As, by \eqref{eq2_M}, $ad-bc \neq 0$, either $b=c=0$ and $a,d \neq 0$, or $a=d=0$ and $b,c \neq 0$.
As $\MM$ is defined up to a factor of proportionality, we can put $a=1$ (resp. $c=1$).
\end{proof}

\begin{corollary}\label{cor5:O1'}
 Let $q\not\equiv0\pmod 3$. In $\PG(3,q)$, for the osculating developable $\Gamma$ of the twisted cubic $\C$ \eqref{eq2_cubic}, the real axes \emph{(}i.e. $\RA$-lines, class $\O_1'=\O_\RA$\emph{)} form an orbit under~$G_q$.
The subgroup of $G_q$ fixing a real axis has size $2(q-1)$.
Moreover the line $\ell=\overline{\Pf(0,0,1,0) \Pf(0,1,0,0)}$ is a real axis.
An element $\Psi$ of $G_q^{\ell}$ has a matrix of the form \eqref{S5:stab_mat_cord}.
\end{corollary}

\begin{proof}
  The first assertions follows from Theorems \ref{th2_Hirs}(iv)(a), \ref{th5_null_pol}, \ref{th5:RC-lines}, and  \cite[Lemma 2.44(ii)]{Hirs_PGFF}. The line  $\ell'=\overline{P_0 P_\infty}$ is a real chord, so by Theorem \ref{th2_Hirs}(iv)(a) the line $\ell' \A = \boldsymbol{\pi}(0,0,0,1) \cap \boldsymbol{\pi}(1,0,0,0) = \overline{\Pf(0,0,1,0) \Pf(0,1,0,0)}$ is a real axis.
Finally we apply Lemma \ref{lemma_same_stab}.
\end{proof}

\begin{lem}\label{lemma:sizeO3orbit}
    For any $q\ge5$, in $\PG(3,q)$, the imaginary chords \emph{(}i.e.\ $\IC$-lines, class $\O_3=\O_\IC$\emph{)} of the twisted cubic $\C$ \eqref{eq2_cubic} form an orbit under~$G_q$. The subgroup of $G_q$ fixing an imaginary chord has size $2(q+1)$.
\end{lem}

\begin{proof} Let $q\equiv\xi\pmod3$.
By Theorem \ref{th2_Hirs}(ii)(b)(c), for $\xi=1$ (resp.  $\xi=0$),  points on imaginary chords form the orbit $\M_4$ (resp. $\M_5$). If $\xi= -1$, points on $\IC$-lines are divided into two orbits $\mathscr{M}_3=\{\t{points on three osculating pla-}$ $\t{nes}\}$ and
    $\mathscr{M}_5=\{\t{points on no osculating plane}\}$. As in $\PG(3,q)$ a plane and a line always meet, for $\xi= -1$ every imaginary chord contains  a point belonging to an osculating plane and therefore to  $\mathscr{M}_3$.

Now, for any $q$, suppose that there exist at least two orbits $\overline{\O}_1$ and  $\overline{\O}_2$ of imaginary chords. Consider IC-lines $\ell_1 \in \overline{\O}_1$ and $\ell_2 \in \overline{\O}_2$. By Theorem \ref{th2_Hirs}(v)(a), no two chords of $\C$ meet off $\C$. Thus, $\ell_1 \cap \ell_2 = \emptyset$ and there exist at least two points  $P_1 \in \ell_1$ and $P_2 \in \ell_2$  belonging to the same orbit; it is $\M_4$, $\M_5$, and $\M_3$ for $\xi=1,0$, and $-1$, respectively. So, there is $\varphi \in G_q$ such that $P_1\varphi = P_2$. By Theorem \ref{th2_Hirs}(iv)(a), $\ell_1\varphi$ is an IC-line. Moreover, by Theorem \ref{th2_Hirs}(v)(a),  every point off $\C$ lies on exactly one chord; thus, $\ell_1\varphi$ is the only imaginary chord containing~$P_2$, i.e. $\ell_1\varphi =\ell_2$. So,  $\overline{\O}_1 = \overline{\O}_2$. Finally by Theorem \ref{th2_Hirs}(iv)(a) and by \cite[Lemma 2.44(ii)]{Hirs_PGFF},
the subgroup of $G_q$ fixing an imaginary chord has size $2(q+1)$.
\end{proof}

\begin{thm}\label{th5:O3orbit}
    For any $q\ge5$, in $\PG(3,q)$, the imaginary chords \emph{(}i.e.\ $\IC$-lines, class $\O_3=\O_\IC$\emph{)} of the twisted cubic $\C$ \eqref{eq2_cubic} form an orbit under~$G_q$. The subgroup of $G_q$ fixing an imaginary chord has size $2(q+1)$.
For $q$ odd, the line $\ell_1=\overline{\Pf(1,0,\rho, 0) \Pf(0,1,0,\rho)}$, $\rho$ a non-square element of $\F_q$, is an $\IC$-line.  An element $\Psi$ of $G_q^{\ell_1}$ has a matrix of the form:
\begin{equation}\label{S5:stab_mat_ICodd}
\MM^{\ell_1}=\left[
 \begin{array}{cccc}
 \a d^3&\a \rho b d^2&\a \rho^2 b^2 d&\a \rho^3 b^3\\
 3bd^2&d^3+2\rho b^2d&\rho^2b^3+2\rho b d^2&3\rho^2b^2d\\
3\a b^2d&\a\rho b^3+2\a bd^2&\a d^3+2\a\rho b^2 d&3\a\rho bd^2\\
 b^3&b^2d&bd^2&d^3
 \end{array}
  \right],
  \end{equation}
$\a \in \{ -1,1\},  b,d \in \F_q,  \t{ not both }0.$

For $q$ even, the line $\ell_2=\overline{\Pf(\eta+1,1,1, 0) \Pf(\eta,\eta,0,1)}$, $\eta$ an element of $\F_q$ of absolute trace $1$, is an $\IC$-line. An element $\Psi$ of $G_q^{\ell_2}$ has a matrix of the form: \begin{equation}\label{S5:stab_mat_ICeven}
   \MM^{\ell_2}=\left[
 \begin{array}{cccc}
A^3&cA^2&c^2A&c^3\\
A^2B&dA^2&c^2B&c^2d\\
AB^2&cB^2&Ad^2&cd^2\\
B^3&dB^2&d^2B&d^3
 \end{array}
  \right],
\end{equation}
$ A =  (\a c+d), B=  (\eta c+(\a+1)d),  \a \in \{ 0,1\},  c,d \in \F_q,  \t{ not both }0.$
\end{thm}

\begin{proof}
The first two assertions follow from Lemma \ref{lemma:sizeO3orbit} . Then, by \cite[Lemma 2.44(ii)]{Hirs_PGFF}, the subgroup of $G_q$ fixing an imaginary chord has size $2(q+1)$.\\
Let $q$ be odd. The line $\ell_1=\overline{\Pf(1,0,\rho, 0) \Pf(0,1,0,\rho)}$, $\rho$ a non-square element of $\F_q$, has equations:
\begin{equation}\label{eqICodd}
\rho x_0 = x_2, \; \rho x_1 = x_3.
\end{equation}
 The coordinate vector of $\ell_1$ \cite[Section 15.2]{Hirs_PG3q} is
\begin{align}\label{eq:coordVectOdd}
  (1,0,\rho,-\rho,0,\rho^2)=\left(\frac{1}{\rho^2},0,\frac{1}{\rho},-\frac{1}{\rho},0,1\right),
\end{align}
Comparing  \eqref{eq1:chord} and \eqref{eq:coordVectOdd}, we obtain:
$a_1 =0, a_2=-\frac{1}{\rho}$. The equation $x^2- a_1x + a_2=x^2-\frac{1}{\rho}=0$ has no solutions as $\rho$ is not a square. Thus, the line  $\ell_1$ is an imaginary chord.\\
The matrix $\MM^{\ell_1}$ of \eqref{S5:stab_mat_ICodd} is obtained from  \eqref{eq2_M} with the substitutions $a = \a d$, $c = \a \rho b$, $\a \in \{ -1,1\}$. The condition $ad-bc \neq 0$ of \eqref{eq2_M} becomes $\a d^2-\a \rho b^2 = d^2- \rho b^2  \neq 0$ and it is always satisfied if $b$ and $d$ are not both $0$, as $\rho$ is not a square. As $\MM^{\ell_1}$ is defined up to a factor of proportionality, we obtain $2(q+1)$ different matrices.
Finally,  $\MM^{\ell_1}$ fixes the line $\ell_1$. In fact:
\begin{equation*}
[1,0,\rho, 0] \MM^{\ell_1}=[
\a d^3+3\a\rho b^2d,
\a \rho^2b^3+3\a\rho bd^2,
\a\rho d^3+3\a \rho^2b^2d,
\a \rho^3b^3+3\a \rho^2bd^2].
\end{equation*}
\begin{equation*}
[0,1,0,\rho] \MM^{\ell_1}=[3bd^2+\rho b^3,  3 \rho b^2d+ d^3, 3 \rho bd^2+\rho^2b^3, 3 \rho^2 b^2 d+\rho d^3].
\end{equation*}
By \eqref{eqICodd}, $[1,0,\rho, 0] \MM^{\ell_1}$, $[0,1,0,\rho] \MM^{\ell_1} \in \ell_1$.

For $q$ even, the line $\ell_2=\overline{\Pf(\eta+1,1,1, 0) \Pf(\eta,\eta,0,1)}$, $\eta$ an element of $\F_q$ of absolute trace $1$, has equations:
\begin{equation}\label{eqICeven}
x_1 + x_2 = \eta x_3, \; x_0 + x_1 = \eta x_2.
\end{equation}
 The coordinate vector of $\ell_2$  is:
\begin{align}\label{eq:coordVectEven1}
  (\eta^2,\eta,\eta+1,\eta,1,1)
\end{align}
Comparing  \eqref{eq1:chord} and \eqref{eq:coordVectEven1}, we obtain:
$a_1=1, a_2=\eta.$
The equation $x^2- a_1x + a_2=x^2+x + \eta=0$ has not solutions,
as the absolute trace of $\eta$ is $1$. Thus, the line $\ell_2$ is an imaginary chord.\\
The matrix $\MM^{\ell_2}$ of \eqref{S5:stab_mat_ICeven} is obtained from  \eqref{eq2_M} with the substitutions $a = \a c+d$, $b = \eta c + (1+\a) d$, $\a \in \{ 0,1\}$. The condition $ad-bc \neq 0$ of \eqref{eq2_M} becomes $d^2+cd+\eta c^2 \neq 0$. If $c = 0$, it implies $d \neq 0$. If $c \neq 0$, by the substitution $x = d/c$, it becomes $x^2+x+\eta \neq 0$ that is always satisfied because the absolute trace of $\eta$ is $1$. As $\MM^{\ell_2}$ is defined up to a factor of proportionality, we obtain $2(q+1)$ different matrices.
Proceeding as in the odd case and performing the symbolic computation using Magma~\cite{Magma}, we verified that $\MM^{\ell_2}$ fixes the line~$\ell_2$.
\end{proof}

\begin{corollary}\label{cor5_O3'}
    Let $q\not\equiv0\pmod 3$. In $\PG(3,q)$, for the osculating developable $\Gamma$ of the twisted cubic $\C$ \eqref{eq2_cubic}, the imaginary axes \emph{(}class $\O_3'=\O_\IA$\emph{)} form an orbit under~$G_q$. The subgroup of $G_q$ fixing an imaginary axis has size $2(q+1)$. For $q$ odd, the line $\ell_1=\overline{\Pf(0,1,0,-3\rho) \Pf(-3/\rho,0,1,0)}$, $\rho$ a non-square, is an imaginary axis.
 An element $\Psi$ of $G_q^{\ell_1}$ has a matrix of the form \eqref{S5:stab_mat_ICodd}.
 For $q$ even, the line $\ell_2=\overline{\Pf(\eta,1,1,0) \Pf(\eta,\eta+1,0,1)}$, $\eta$ an element of absolute trace $1$, is an imaginary axis.
 An element $\Psi$ of $G_q^{\ell_2}$, has a matrix of the form \eqref{S5:stab_mat_ICeven}.
\end{corollary}

\begin{proof}

The first assertions follow from Theorems \ref{th2_Hirs}(iv)(a), \ref{th5_null_pol}, \ref{th5:O3orbit}.  By Theorem \ref{th5:O3orbit}, for odd $q$ (resp. for even $q$)
the line  $\ell_1=$ $\overline{\Pf(1,0,\rho, 0) \Pf(0,1,0,\rho)}$, $\rho$ a non-square element, (resp.  the line $\ell_2=\overline{\Pf(\eta+1,1,1, 0) \Pf(\eta,\eta,0,1)}$, $\eta$ an element of absolute trace $1$) is an $\IC$-line, so by Theorem \ref{th2_Hirs}(iv)(a) the line $\ell_1 \A = \boldsymbol{\pi}(0,-3\rho,0,-1) \cap \boldsymbol{\pi}(\rho,0,3,0) = $
 $\overline{\Pf(0,1,0,-3\rho) \Pf(-3/\rho,0,1,0)}$
 (resp. the line $\ell_2 \A = \boldsymbol{\pi}(0,1,1,\eta+1) \cap \boldsymbol{\pi}(1,0,\eta,\eta) = $
 $\overline{\Pf(\eta,1,1,0) \Pf(\eta,\eta+1,0,1})$ is an $\IA$-line.
Finally we apply Lemma \ref{lemma_same_stab}.
\end{proof}

\section{Orbits under $G_q$ of non-tangent unisecants and external lines with respect to the cubic~$\C$
}\label{sec:orbUGUnGEG}

\begin{lem}\label{lm:l1inUG}
For any $q\ge5$, let $\ell_1=\overline{P_0\Pf(0,1,0,0)}$,   $\ell_2=\overline{P_0\Pf(0,1,1,0)}$, $P_0=\Pf(0,0,0,1)\in\C$. Then  $\ell_1$ and  $\ell_2$ are non-tangent unisecants in a $\Gamma$-plane, i.e. $\ell_1, \ell_2 \in \UG$.
\end{lem}

\begin{proof}
By \eqref{eq2_plane}, \eqref{eq2_osc_plane}, $\pi_\t{osc}(0)$ has equation $x_0 = 0$. By Lemma \ref{eqTangents}, the tangent $\T_0$ to $\C$ at $P_0$ has equations $x_0 = x_1 = 0$. Therefore  $\ell_1, \ell_2 \in \pi_\t{osc}(0)$ and  $\ell_1, \ell_2 \neq \T_0$.
\end{proof}

\begin{thm}\label{th5:O4UGorbitOdd}
For any $q\ge5$, $q$ odd, in $\PG(3,q)$, for the twisted cubic $\C$ of \eqref{eq2_cubic}, the non-tangent unisecants in a $\Gamma$-plane \emph{(}i.e. $\UG$-lines, class $\O_4=\O_\UG$\emph{)} form an orbit under~$G_q$.
The subgroup of $G_q$ fixing a  $\UG$-line has size $q-1$.
Consider the $\UG$-line $\ell=\overline{P_0\Pf(0,1,0,0)}$. An element $\Psi$ of $G_q^{\ell}$ has a matrix of the form:
 \begin{equation*}\label{S5:stab_mat}
   \MM^{\ell}=\left[
 \begin{array}{cccc}
 1&0&0&0\\
 0&d&0&0\\
 0&0&d^2&0\\
 0&0&0&d^3
 \end{array}
  \right], \;\;d \in \F_q^*.
\end{equation*}
\end{thm}

\begin{proof}
Let  $\ell=\overline{P_0\Pf(0,1,0,0)}$. The line $\ell$ has equations $x_0 = 0, x_2 = 0$.  By Lemma \ref{lm:l1inUG}, $\ell,  \in \UG$. We compute the subgroup of $G_q$ fixing $\ell$.
Let  $\Psi\in G_q$ and let $\MM$ be a matrix corresponding to $\Psi$. By \eqref{eq2_M},
\begin{equation*}\label{ThS5:stab_eq1}
[0,0,0,1]\MM=[b^3,b^2d,bd^2,d^3].
\end{equation*}
Then $P_0 \Psi\in \ell_1$ implies $b = 0$ and $a,d \neq 0$, as $ad-bc \neq 0$. On the other hand,
\begin{equation*}
[0,1,0,0]\MM=[0,a^2d,2acd,3c^2d].
\end{equation*}
Then $\Pf(0,1,0,0)  \Psi\in \ell$ implies $c = 0$. As $\MM$ is defined up to a factor of proportionality, we can put $a=1$, so $\#G_q^{\ell}=q-1$. By  \cite[Lemma 2.44(ii)]{Hirs_PGFF}, the orbit of $\ell$ under $G_q$ has size $\#G_q / \#G_q^{\ell}= q(q+1) = \#\O_\UG$.
\end{proof}

\begin{thm}\label{th5:O4UGorbitEven}
For any $q\ge8$, $q$ even, in $\PG(3,q)$, for the twisted cubic $\C$ of \eqref{eq2_cubic},
 the non-tangent unisecants in a $\Gamma$-plane \emph{(}i.e. $\UG$-lines, class $\O_4=\O_\UG$\emph{)} form two orbits of size $q+1$ and $q^2-1$
that can be represented in the form $\{\ell_1\varphi|\varphi\in G_q\}$ and $\{\ell_2\varphi|\varphi\in G_q\}$, respectively, where $\ell_1=\overline{P_0\Pf(0,1,0,0)}$, $\ell_2=\overline{P_0\Pf(0,1,1,0)}$, $P_0=\Pf(0,0,0,1)\in\C$. Moreover the orbit of size $q+1$ consists of the lines in the regulus complementary to that of the tangents.
The group $G_q^{\ell_1}$ has size $q(q-1)$. An element $\Psi$ of $G_q^{\ell_{1}}$ has a matrix of the form:
 \begin{equation}\label{S5:stab_mat}
   \MM^{\ell_{1}}=\left[
 \begin{array}{cccc}
 1&c&c^2&c^3\\
 0&d&0&c^2d\\
 0&0&d^2&cd^2\\
 0&0&0&d^3
 \end{array}
  \right], \;\;c \in \F_q, d \in \F_q^*.
\end{equation}
The group $G_q^{\ell_2}$ has size $q$. An element $\Psi$ of $G_q^{\ell_{2}}$ has a matrix of the form:
 \begin{equation}\label{S5:stab_mat}
   \MM^{\ell_{2}}=\left[
 \begin{array}{cccc}
 1&c&c^2&c^3\\
 0&1&0&c^2\\
 0&0&1&c\\
 0&0&0&1
 \end{array}
  \right], \;\;c \in \F_q.
\end{equation}
\end{thm}

\begin{proof}
 By Lemma \ref{lm:l1inUG}, $\ell_1, \ell_2, \in \UG$.
 The line $\ell_1$ has equations $x_0 = 0, x_2 = 0$.  We compute the subgroup of $G_q$ fixing $\ell_1$.
Let  $\Psi\in G_q$ and let $\MM$ be a matrix corresponding to $\Psi$. As in Theorem \ref{th5:O4UGorbitOdd}, $P_0 \Psi\in \ell_1$ implies $b = 0$ and $a,d \neq 0$. On the other hand,
\begin{equation*}
[0,1,0,0]\MM=[0,a^2d,0,c^2d].
\end{equation*}
Then $\Pf(0,1,0,0)  \Psi\in \ell_1$. As $\MM$ is defined up to a factor of proportionality, we can put $a=1$, so $\#G_q^{\ell_1}=q(q-1)$. By  \cite[Lemma 2.44(ii)]{Hirs_PGFF}, the orbit of $\ell_1$ under $G_q$ has size $\#G_q / \#G_q^{\ell_1}= q+1$. On the content of this orbit, see Theorem \ref{th2_Hirs}(iv)(a).

The line $\ell_2$ has equations $x_0 = 0, x_1 = x_2$.  We compute the subgroup of $G_q$ fixing $\ell_2$.
Let  $\Psi\in G_q$ and let $\MM$ be a matrix corresponding to $\Psi$. Again $P_0 \Psi\in \ell_2$ implies $b = 0$ and $a,d \neq 0$. On the other hand,
\begin{equation*}
[0,1,1,0]\MM=[0,a^2d,ad^2,cd^2+cd^2].
\end{equation*}
Then $\Pf(0,1,1,0)  \Psi\in \ell_2$ implies $a=d$. As $\MM$ is defined up to a factor of proportionality, we can put $a=1$, so $\#G_q^{\ell_2}=q$. By  \cite[Lemma 2.44(ii)]{Hirs_PGFF}, the orbit of $\ell_2$ under $G_q$ has size $ q^2-1$. The two orbits are distinct and by Theorem \ref{th2_Hirs}(iv)(a) they form a partition of $\O_\UG$.
\end{proof}

\begin{lem}\label{lm:l1inUnG}
For any $q\ge5$, let $\ell_1=\overline{P_0\Pf(1,0,1,0)}$,  $P_0=\Pf(0,0,0,1)\in\C $and for $q$ odd let $\ell_2=\overline{P_0\Pf(1,0,\rho,0)}$,  $\rho$ is not a square. Then $\ell_1$ and $\ell_2$ are non-tangent unisecants not in a $\Gamma$-plane, i.e. $\ell_1, \ell_2 \in \UnG$.
\end{lem}

\begin{proof}
 By Lemma \ref{eqTangents}, the tangent $\T_0$ to $\C$ at $P_0$ has equations $x_0 = x_1 = 0$, so $\ell_1, \ell_2 \neq \T_0$.
Also, $\ell_1, \ell_2$ are not real chords, as $\ell_1$ has equations $x_0 = x_2 , x_1 = 0$ and $\ell_2$ has equations $\rho x_0 = x_2 , x_1 = 0$.
Finally by \eqref{eq2_plane}, \eqref{eq2_osc_plane}, $\pi_\t{osc}(0)$ has equation $x_0 = 0$, so  $\ell_1, \ell_2 \not\subset  \pi_\t{osc}(0)$.
\end{proof}

\begin{thm}\label{th5:O5UnGorbitEven}
Let $q\ge8$, $q$ even. In $\PG(3,q)$, for the twisted cubic $\C$ of \eqref{eq2_cubic}, the non-tangent unisecants not in a $\Gamma$-plane \emph{(}i.e. $\UnG$-lines, class $\O_5=\O_\UnG$\emph{)} form an orbit under~$G_q$. The orbit  can be represented in the form $\{\ell \varphi|\varphi\in G_q\}$, where $\ell=\overline{P_0\Pf(1,0,1,0)}$, $P_0=\Pf(0,0,0,1)\in\C$. The subgroup of $G_q$ fixing any $\UG$-line is trivial.
\end{thm}

\begin{proof}
We argue similar to the proof of Theorem \ref{th5:O4UGorbitOdd}.
Consider the line  $\ell=\overline{P_0\Pf(1,0,1,0)}$. The line $\ell$ has equations $x_1 = 0, x_0 = x_2$.  By Lemma \ref{lm:l1inUnG}, $\ell  \in \UnG$. We compute the subgroup of $G_q$ fixing $\ell$.
Let  $\Psi\in G_q$ and let $\MM$ be a matrix corresponding to $\Psi$.
\begin{equation*}\label{ThS5:stab_eq1}
[0,0,0,1]\MM=[b^3,b^2d,bd^2,d^3].
\end{equation*}
Then $P_0 \Psi\in \ell$ implies $b^2d = 0$. If $d=0$, $P_0 \Psi = \Pf(1,0,0,0) \notin \ell$, contradiction. Therefore $b = 0$ and $a,d \neq 0$, as $ad-bc \neq 0$. On the other hand,
\begin{equation*}
[1,0,1,0]\MM=[a^3,a^2c,ac^2+ad^2,c^3+cd^2].
\end{equation*}
Then $\Pf(1,0,1,0)  \Psi\in \ell$ implies $c = 0$ and $a^2=d^2$. As $q$ is even, this implies $d=a$. As $\MM$ is defined up to a factor of proportionality, we can put $a=1$, so $\MM$ is the identity matrix and $\#G_q^{\ell}=1$. By  \cite[Lemma 2.44(ii)]{Hirs_PGFF}, the orbit of $\ell$ under $G_q$ has size $q^3-q = \#\O_\UnG$.
\end{proof}

\begin{corollary}\label{cor5:O5'EGeven}
Let $q\not\equiv0\pmod 3$, $q$ even. In $\PG(3,q)$, for the twisted cubic $\C$ of \eqref{eq2_cubic},
    the external lines in a $\Gamma$-plane   \emph{(}class $\O_5'=\O_\EG$\emph{)} form an orbit under~$G_q$. The line $\ell = \overline{\Pf(0,0,1,0) \Pf(0,1,0,1)}$ is an $\EG$-line.
 The subgroup of $G_q$ fixing any $\EG$-line is trivial.
\end{corollary}
\begin{proof}
The first assertion follows from Theorems \ref{th2_Hirs}(iv)(a), \ref{th5_null_pol}, \ref{th5:O5UnGorbitEven}.  By Lemma \ref{lm:l1inUnG},
The line $\ell'=\overline{P_0\Pf(1,0,1,0)}$, $P_0=\Pf(0,0,0,1)$ is a $\UnG$-line, so the line $\ell' \A = \boldsymbol{\pi}(0,0,0,1) \cap \boldsymbol{\pi}(1,0,1,0) = $
 $\overline{\Pf(0,0,1,0) \Pf(0,1,0,1)}$ is an $\EG$-line by Theorem \ref{th2_Hirs}(iv)(a).
Finally we apply Lemma \ref{lemma_same_stab}.
\end{proof}

\begin{thm}\label{th5:O5UnGorbitOdd}
Let $q\ge5$, $q$ odd. In $\PG(3,q)$, for the twisted cubic $\C$ of \eqref{eq2_cubic},
 the non-tangent unisecants not in a $\Gamma$-plane \emph{(}i.e. $\UnG$-lines, class $\O_5=\O_\UnG$\emph{)} form two orbits of size $(q^3-q)/2$. The two orbits can be represented in the form $\{\ell_j\varphi|\varphi\in G_q\}$, $j=1,2$, where $\ell_1=\overline{P_0\Pf(1,0,1,0)}$,  $\ell_2=\overline{P_0\Pf(1,0,\rho,0)}$, $P_0=\Pf(0,0,0,1)\in\C$, $\rho$ is not a square. Moreover, the $\UnG$-lines $\ell_1$ and $\ell_2$ are fixed by the same subgroup $G_q^{\ell_1, \ell_2}$ of $G_q$ of size $2$. An element $\Psi$ of $G_q^{\ell_1, \ell_2}$ has a matrix of the form:
 \begin{equation}\label{S5:stab_matUnGodd}
   \MM^{\ell_1, \ell_2}=\left[
 \begin{array}{cccc}
 1&0&0&0\\
 0&d&0&0\\
 0&0&d^2&0\\
 0&0&0&d^3
 \end{array}
  \right], \;\;d \in \{1,-1\}.
\end{equation}

\end{thm}

\begin{proof}
We argue similar to the proof of Theorem \ref{th5:O4UGorbitEven}.
By Lemma \ref{lm:l1inUnG}, $\ell_1, \ell_2, \in \UnG$.
 The line $\ell_1$ has equations $x_1 = 0, x_0 = x_2$.  We compute the subgroup of $G_q$ fixing $\ell_1$.
Let  $\Psi\in G_q$ and let $\MM$ be a matrix corresponding to $\Psi$. As in Theorem \ref{th5:O4UGorbitEven}, $P_0 \Psi\in \ell_1$ implies $b = 0$ and $a,d \neq 0$. On the other hand,
\begin{equation*}
[1,0,1,0]\MM=[a^3,a^2c,ac^2+ad^2,c^3+3cd^2].
\end{equation*}
Then $\Pf(1,0,1,0)  \Psi\in \ell_1$ implies c=0 and $d^2=a^2$. As $\MM$ is defined up to a factor of proportionality, we can put $a=1$, so $\#G_q^{\ell_1}=2$. By  \cite[Lemma 2.44(ii)]{Hirs_PGFF}, the orbit of $\ell_1$ under $G_q$ has size $\#G_q / \#G_q^{\ell_1}= (q^3-q)/2$.

The line $\ell_2$ has equations $x_1 = 0, \rho x_0 = x_2$.  We compute the subgroup of $G_q$ fixing $\ell_2$.
Let  $\Psi\in G_q$ and let $\MM$ be a matrix corresponding to $\Psi$. Again $P_0 \Psi\in \ell_2$ implies $b = 0$ and $a,d \neq 0$. On the other hand,
\begin{equation*}
[1,0,\rho,0]\MM=[a^3,a^2c,ac^2+\rho ad^2,c^3+3\rho cd^2].
\end{equation*}
Then $\Pf(1,0,1,0)  \Psi\in \ell_1$  implies $c=0$ and $\rho a^2=\rho d^2$. As $\MM$ is defined up to a factor of proportionality, we can put $a=1$, so $\#G_q^{\ell_2}=2$. By  \cite[Lemma 2.44(ii)]{Hirs_PGFF}, the orbit of $\ell_2$ under $G_q$ has size $(q^3-q)/2$.

The lines $\ell_1$ and $\ell_2$ belong to different orbits. Suppose there exists  $\Psi\in G_q$ such that $\ell_1 \Psi = \ell_2$ and let $\MM$ be a matrix corresponding to $\Psi$.
Then
\begin{equation*}
[0,0,0,1]\MM=[b^3,b^2d,bd^2,d^3]
\end{equation*}
and $P_0 \Psi\in \ell_2$ implies $b^2d = 0$, $\rho b^3=bd^2$. If $b \neq 0$, $\rho b^2=d^2$, contradiction as $\rho$ is not a square.  If $b = 0$, then $a,d \neq 0$. Consider
\begin{equation*}
[1,0,\rho,0]\MM=[a^3,a^2c,ac^2+\rho ad^2,c^3+3\rho cd^2].
\end{equation*}
Then $\Pf(1,0,1,0)  \Psi\in \ell_2$  implies $c=0$ and $\rho a^2= d^2$,  contradiction as $\rho$ is not a square.
Therefore the two orbits are distinct and by Theorem \ref{th2_Hirs}(iv)(a) they form a partition of $\O_\UnG$.
\end{proof}

\begin{corollary}\label{cor5:O5'EGodd}
Let $q\not\equiv0\pmod 3$, $q$ odd. In $\PG(3,q)$, for the twisted cubic $\C$ of \eqref{eq2_cubic},
    the external lines in a $\Gamma$-plane   \emph{(}class $\O_5'=\O_\EG$\emph{)} form  two orbits of size $(q^3-q)/2$. The two orbits can be represented in the form $\{\ell_j\varphi|\varphi\in G_q\}$, $j=1,2$, where $\ell_j=\pk_0\cap\pk_j$ is the intersection line of planes $\pk_0$ and $\pk_j$ such that
           $\pk_0=\boldsymbol{\pi}(1,0,0,0)=\pi_\t{\emph{osc}}(0)$, $\pk_1=\boldsymbol{\pi}(0,-3,0,-1)$,  $\pk_2=\boldsymbol{\pi}(0,-3\rho,0,-1)$, $\rho$ is not a square, cf. \eqref{eq2_plane}, \eqref{eq2_osc_plane}.
           Moreover, the $\EG$-lines $\ell_1$ and $\ell_2$ are fixed by the same subgroup $G_q^{\ell_1, \ell_2}$ of $G_q$ of size $2$. An element $\Psi$ of $G_q^{\ell_1, \ell_2}$ has a matrix of the form
 \eqref{S5:stab_matUnGodd}.
\end{corollary}

\begin{proof}
The first assertion follows from Theorems \ref{th2_Hirs}(iv)(a), \ref{th5_null_pol} and \ref{th5:O5UnGorbitOdd}.
  The null polarity $\A$ \eqref{eq2_null_pol}
  maps the points $P_0=\Pf(0,0,0,1)$, $P'=\Pf(1,0,1,0)$, and $P''=\Pf(1,0,\rho,0)$
 of Theorem \ref{th5:O5UnGorbitOdd}
to the planes $\pk_0=\boldsymbol{\pi}(1,0,0,0)$, $\pk_1=\boldsymbol{\pi}(0,-3,0,-1)$, and $\pk_2=\boldsymbol{\pi}(0,-3\rho,0,-1)$, respectively. The $\UnG$-lines $\ell'=\overline{P_0P'}$ and  $\ell''=\overline{P_0P''}$ are mapped to $\EG$-lines so that $\ell'\A=\pk_0\cap\pk_1\triangleq\ell_1$ and $\ell''\A=\pk_0\cap\pk_2\triangleq\ell_2$.
Finally we apply Lemma \ref{lemma_same_stab}.
\end{proof}

\section{Orbits of $\EA$-lines under $G_q$, $q\equiv 0 \pmod 3$}\label{sec:orbEA}

In the following we consider $q\equiv 0 \pmod 3$, $q \geq 9$, and denote by $\ell_{\Ar}$  the  axis of $\Gamma$, by  $P_{\infty}^{\Ar}$ the point  $\Pf(0,0,1,0)$  of $\ell_{\Ar}$ and by $P_t^{\Ar}$ the point $\Pf(0,1,t,0)$  of $\ell_{\Ar}$ with $t\in\F_q$. The line $\ell_{\Ar}$ is the intersection of the osculating planes, so has equations $x_0 =  x_3 = 0$, and $P_{\Ar} \in \ell_{\Ar}$. Recall that by Theorem \ref{th2_Hirs}(iv)(b),  $\ell_{\Ar}$ is fixed by  $G_q$.


\begin{proposition}\label{lem5:GqtranslA}
\looseness - 1    Let $q\equiv0\pmod 3$, $q \geq 9$. The group $G_q$ acts transitively on $\ell_{\Ar}$.
\end{proposition}

\begin{proof}
If we take $\varphi \in G_q$ whose matrix in the form \eqref{eq2_M} has $a=0$, $b=c=d=1$, then $P_0^{\Ar} \varphi =\Pf (0, 0, 1,0)=P_{\infty}^{\Ar}$.
If we take $\varphi \in G_q$ whose matrix in the form \eqref{eq2_M} has $a=d=1, b=0, c=-n$, then $ P_0^{\Ar} \varphi =\Pf (0, 1, n,0) = P_n^{\Ar}$.
\end{proof}

\begin{lem}\label{lm:l1inEA}
For any $q\equiv0\pmod 3$, $q \geq 9$, let $\ell_1=\overline{P_0^\Ar\Pf(0,0,1,1)}$,  $\ell_2=\overline{P_0^\Ar\Pf(1,0,1,0)}$, $\ell_3=\overline{P_0^\Ar\Pf(1,0,\rho,0)}$, $P_0^\Ar=\Pf(0,1,0,0)$, $\rho$ is not a square. Then $\ell_1, \ell_2, \ell_3 \in \EA$.
\end{lem}

\begin{proof}
As pairs of equations defining $\ell_1$, $\ell_2$, $\ell_3$ are $x_0 = 0, x_2=x_3$;
 $x_3 = 0$, $x_0=x_2$; $x_3 = 0, \rho x_0=x_2$, respectively,  $\ell_i \cap \C = \emptyset$ and  $\ell_i \cap \ell_{\Ar} = P_0^\Ar, i = 1,2,3.$
\end{proof}

\begin{thm}\label{th5:EAorbit}
For any $q\equiv0\pmod 3,\; q\ge9$, in $\PG(3,q)$, for the twisted cubic $\C$ of \eqref{eq2_cubic},
 the external lines meeting the axis of $\Gamma$ \emph{(}i.e. $\EA$-lines, class $\O_8=\O_\EA$\emph{)} form three orbits under~$G_q$  of size $q^3-q, (q^2-1)/2, (q^2-1)/2$. The $(q^3-q)$-orbit and the two $(q^2-1)/2$-orbits can be represented in the form $\{\ell_1\varphi|\varphi\in G_q\}$ and $\{\ell_j\varphi|\varphi\in G_q\}$, $j=2,3$, respectively, where $\ell_1=\overline{P_0^\Ar\Pf(0,0,1,1)}$,  $\ell_2=\overline{P_0^\Ar\Pf(1,0,1,0)}$, $\ell_3=\overline{P_0^\Ar\Pf(1,0,\rho,0)}$, $P_0^\Ar=\Pf(0,1,0,0)$, $\rho$ is not a square.
The group $G_q^{\ell_1}$  is trivial. The $\EA$-lines $\ell_2$ and $\ell_3$ are fixed by the same subgroup $G_q^{\ell_2,\ell_3}$  of $G_q$ of size $2q$. An element $\Psi$ of $G_q^{\ell_2,\ell_3}$  has a matrix of the form:
 \begin{equation}\label{S5:stab_mat}
   \MM^{\ell_2,\ell_3}=\left[
 \begin{array}{cccc}
 1&0&0&0\\
 0&d&0&0\\
 0&-bd&d^2&0\\
 b^3&b^2d&bd^2&d^3
 \end{array}
  \right], \;\;b \in \F_q, d \in \{1,-1\}.
\end{equation}

\end{thm}

\begin{proof}

By Lemma \ref{lm:l1inEA}, $\ell_1, \ell_2, \ell_3 \in \EA$.
 The line $\ell_1$ has equations $x_0 = 0, x_2 = x_3$.  We compute the subgroup of $G_q$ fixing $\ell_1$.
Let  $\Psi\in G_q$ and let $\MM$ be a matrix corresponding to $\Psi$.  By \eqref{eq2_M},
\begin{equation}\label{ImPA}
[0,1,0,0]\MM=[0,a^2d-abc,bc^2-acd,0].
\end{equation}
Then $P_0^\Ar \Psi\in \ell_1$ implies $bc^2-acd = 0$. If $c \neq 0$, $bc-ad = 0$, contradiction, so $c = 0$ and $a,d \neq 0$. On the other hand,
\begin{equation*}
[0,0,1,1]\MM=[b^3,-abd+b^2d,ad^2+bd^2,d^3].
\end{equation*}
Then $\Pf(0,0,1,1)  \Psi\in \ell_1$ implies b=0 and $d=a$. As $\MM$ is defined up to a factor of proportionality, we can put $a=1$, so $\MM$ is the identity matrix and $\#G_q^{\ell_1}=1$. By  \cite[Lemma 2.44(ii)]{Hirs_PGFF}, the orbit of $\ell_1$ under $G_q$ has size $\#G_q / \#G_q^{\ell_1}= q^3-q$.

The line $\ell_2$ has equations $x_3 = 0, x_0 = x_2$.  We compute the subgroup of $G_q$ fixing $\ell_2$.
Let  $\Psi\in G_q$ and let $\MM$ be a matrix corresponding to $\Psi$. By \eqref{ImPA}, $P_0^\Ar \Psi\in \ell_2$ implies $c = 0$ and $a,d \neq 0$. On the other hand,
\begin{equation} \label{Im1010}
[1,0,1,0]\MM=[a^3,-abd,ad^2,0].
\end{equation}
Then $\Pf(1,0,1,0)  \Psi\in \ell_2$  implies $a^2=d^2$. As $\MM$ is defined up to a factor of proportionality, we can put $a=1$, so $\#G_q^{\ell_2}=2q$. By  \cite[Lemma 2.44(ii)]{Hirs_PGFF}, the orbit of $\ell_2$ under $G_q$ has size $(q^2-1)/2$.

The line $\ell_3$ has equations $x_3 = 0, \rho x_0 = x_2$.  We compute the subgroup of $G_q$ fixing $\ell_3$.
Let  $\Psi\in G_q$ and let $\MM$ be a matrix corresponding to $\Psi$. By \eqref{ImPA}, $P_0^\Ar \Psi\in \ell_3$ implies $c = 0$ and $a,d \neq 0$. On the other hand,
\begin{equation*}
[1,0,\rho,0]\MM=[a^3,-\rho abd,\rho ad^2,0].
\end{equation*}
Then $\Pf(1,0,\rho,0)  \Psi\in \ell_3$ implies $a^2=d^2$. As $\MM$ is defined up to a factor of proportionality, we can put $a=1$, so $\#G_q^{\ell_3}=2q$. By  \cite[Lemma 2.44(ii)]{Hirs_PGFF}, the orbit of $\ell_3$ under $G_q$ has size $(q^2-1)/2$.

The lines $\ell_2$ and $\ell_3$ belong to different orbits. Suppose there exists  $\Psi\in G_q$ such that $\ell_2 \Psi = \ell_3$ and let $\MM$ be a matrix corresponding to $\Psi$.
Then by \eqref{ImPA}, $P_0^\Ar \Psi\in \ell_3$ implies $c = 0$ and $a,d \neq 0$. On the other hand,  by \eqref{Im1010}, $\Pf(1,0,1,0) \Psi\in \ell_3$ implies  $\rho a^2=d^2$, contradiction as $\rho$ is not a square.
Therefore the three orbits are distinct and by Theorem \ref{th2_Hirs}(iv)(b) they form a partition of $\O_\EA$.
\end{proof}

\section{Open problems for $\EnG$-lines and their solutions for $5\le q\le37$ and $q=64$}\label{sec:classification}

Let $q\equiv\xi\pmod3$ with $\xi\in\{-1,0,1\}$. We introduce sets $Q^{(\xi)}_\bullet$ of $q$ values with the natural subscripts ``$\mathrm{od}$'' (odd) and ``$\mathrm{ev}$" (even).
\begin{align*}
        &Q^{(0)}_\t{od}=\{9,27\},~Q^{(1)}_\mathrm{od}=\{7,13,19,25,31,37\},~Q^{(-1)}_\mathrm{od}=\{5,11,17,23,29\};\db\\
             &Q_\t{ev}=\{8,16,32,64\}.
    \end{align*}
Theorem \ref{th8_MAGMA_odd}  has been proved by an  exhaustive computer search using the symbol calculation system Magma \cite{Magma}.

\begin{thm} \label{th8_MAGMA_odd}For $q\in Q^{(1)}_\mathrm{od}\cup Q^{(-1)}_\mathrm{od}\cup Q^{(0)}_\mathrm{od}$ and $q\in Q_\t{\emph{ev}}$, all the results of Sections \emph{\ref{sec_mainres}--\ref{sec:orbEA}} are confirmed by computer search. Also, the following holds.
\begin{description}
  \item[(i)]
Let $q\equiv\xi\pmod3$,~$\xi\in\{1,-1,0\}$.  Let $q\in Q^{(1)}_\mathrm{od}\cup Q^{(-1)}_\mathrm{od}\cup Q^{(0)}_\mathrm{od}$ be odd. Then we have the following:

The total number of $\EnG$-line orbits is  $2q-3+\xi$.

The total number of line orbits in $\PG(3,q)$ is $2q+7+\xi$.

 Under $G_q$, for $\EnG$-lines with $\xi\in\{1,-1,0\}$, there are\\
 $
     \begin{array}{lcl}
        n_q(\xi)&\t{orbits of length}& (q^3-q)/4, \\
       q-1&\t{orbits of length}&(q^3-q)/2,\\
       (q-\xi)/3&\t{orbits of length}&q^3-q,
     \end{array}
     $\\
     where $n_q(1)=(2q-11)/3$ , $n_q(-1)=(2q-10)/3$,
    $n_q(0)=(2q-6)/3$.\\
     In addition, for $q\in Q^{(1)}_\t{\emph{od}}$, there are\\
     $
     \begin{array}{lcl}
       1 &\t{orbit of length}&(q^3-q)/12, \\
       2&\t{orbits of length}&(q^3-q)/3.
     \end{array}
     $

  \item[(ii)]
  Let $q\equiv\xi\pmod3$, $\xi\in\{1,-1\}$.  Let $q\in Q_\t{\emph{ev}}$ be even. Then we have the following:

The total number of $\EnG$-line orbits is  $2q-2+\xi$.

The total number of line orbits in $\PG(3,q)$ is $2q+7+\xi$.

    Under $G_q$, for $\EnG$-lines, there are

     $2+\xi$  orbits of length $(q^3-q)/(2+\xi)$;

     $2q-4$ orbits of length $(q^3-q)/2$.
 \end{description}
\end{thm}

\begin{conjecture}
The results of Theorem \emph{\ref{th8_MAGMA_odd}} hold for all $q\ge5$ with the corresponding parity and $\xi$ value.
\end{conjecture}

\noindent \textbf{Open problems.}
Find the number, sizes and the structures of orbits of the class $\O_6=\O_\EnG$ (i.e. external lines, other than a chord, not in a $\Gamma$-plane).
Prove the corresponding results of Theorem \ref{th8_MAGMA_odd} for all $q\ge5$.

\section*{Acknowledgments}
The research of S. Marcugini, and F. Pambianco was
supported in part by the Italian
National Group for Algebraic and Geometric Structures and
their Applications (GNSAGA -
INDAM) (Contract No. U-UFMBAZ-2019-000160, 11.02.2019) and
by University of Perugia
(Project No. 98751: Strutture Geometriche, Combinatoria e
loro Applicazioni, Base Research
Fund 2017-2019).

\end{document}